\documentclass[12pt]{amsart}
\usepackage{amssymb,amsmath,amsfonts,amsthm,nomencl,mathrsfs}
\usepackage[arrow, matrix, curve]{xy}
\usepackage[latin1]{inputenc}
\usepackage{a4wide}
\DeclareMathOperator{\supp}{supp}

\newcommand{\IC}{\mathbb{C}}
\newcommand{\IR}{\mathbb{R}}

\newcommand{\question}[1]{\leavevmode{\marginpar{\tiny%
$\hbox to 0mm{\hspace*{-0.5mm}$\leftarrow$\hss}%
\vcenter{\vrule depth 0.1mm height 0.1mm width \the\marginparwidth}%
\hbox to 0mm{\hss$\rightarrow$\hspace*{-0.5mm}}$\\\relax\raggedright #1}}}

\newcommand{\ILL}{\mathscr{L}}

\newcommand{\IEE}{E}

\newcommand{\IHH}{\mathscr{H}}

\newcommand{\ISS}{\mathscr{S}}

\newcommand{\dom}{\mathrm{Dom}}

\newcommand{\loc}{\mathrm{loc}}

\newcommand{\IN}{\mathbb{N}}

\newcommand{\IX}{\mathbb{X}}
\newcommand{\IP}{\mathbb{P}}

\newcommand{\pa}{\slash\slash}

\newcommand{\Id}{  d}

\theoremstyle{plain}            
\newtheorem{theorem}{theorem}[section]
\newtheorem{Lemma}[theorem]{Lemma}
\newtheorem{Corollary}[theorem]{Corollary}
\newtheorem{Theorem}[theorem]{Theorem}
\newtheorem{Proposition}[theorem]{Proposition}

\theoremstyle{definition}
\newtheorem{Definition}[theorem]{Definition}

\newtheorem{Notation}[theorem]{Notation}

\newtheorem{Remark}[theorem]{Remark}
\newtheorem{Example}[theorem]{Example}
\allowdisplaybreaks[1]

\begin{document}
\title[Kac regular sets and Sobolev spaces]{Kac regular sets and Sobolev spaces in geometry, probability and quantum physics}

\author[F. Bei]{Francesco Bei}
\address{Francesco Bei, Dipartimento di Matematica, Sapienza Universit\`a di Roma, 00185 Roma, Italy}
\email{ bei@mat.uniroma.it}

\author[B. G\"uneysu]{Batu G\"uneysu}
\address{Batu G\"uneysu, Institut f\"ur Mathematik, Humboldt-Universit\"at zu Berlin, 12489 Berlin, Germany} \email{gueneysu@math.hu-berlin.de}

\maketitle

\begin{abstract} Let $\Omega\subset M$ be an open subset of a Riemannian manifold $M$ and let $V:M\to \IR$ be a Kato decomposable potential. With $W^{1,2}_{0}(M;V)$ the natural form domain of the Schr\"odinger operator $-\Delta+V$ in $L^2(M)$, in this paper we study systematically the following question: Under which assumption on $\Omega$ is the statement
$$
\text{ for all $f\in W^{1,2}_{0}(M;V)$ with $f=0$ a.e. in $M\setminus \Omega$ one has $f|_\Omega\in W^{1,2}_{0}(\Omega;V)$}
$$
true for every such $V$? Generalizing a classical result by Herbst and Zhao, who treat the Euclidean $\IR^m$ and $V=0$, we prove that without any further assumptions on $V$, the above property is satisfied, if $\Omega$ is Kac regular, a probabilistic property which means that the first exit time of Brownian motion on $M$ from $\Omega$ is equal to its first penetration time to $M\setminus \Omega$. In fact, we treat more general covariant Schr\"odinger operators acting on sections in metric vector bundles, allowing new results concerning the harmonicity of Dirac spinors on singular subsets. Finally, we prove that locally Lipschitz regular $\Omega$'s are Kac regular.   
\end{abstract}

\section{Introduction}

Consider the following  three properties that an open subset $\Omega$ of a noncompact Riemannian manifold $M$ may or may not have: \\
(A) Let $V$ be a Kato decomposable potential on $M$ and let $H_M(V)$ denote the natural self-adjoint realization of $-\Delta+V$ in $L^2(M)$ and let $H_{\Omega}(V)$ denote the self-adjoint realization of $-\Delta+V$ in $L^2(\Omega)$ subject to Dirichlet boundary conditions (also the Friedrichs realization). \emph{For every such $V$ one has
\begin{align}
\exp\big(-t(H_M(V)+\infty\cdot 1_{M\setminus \Omega})\big)=\exp(-t H_{\Omega}(V))P_\Omega\quad\text{ for all $t>0$,}
\end{align}
where $\exp\big(-t(H_M(V)+\infty\cdot 1_{M\setminus \Omega})\big)$ is understood as the strong limit of $\exp\big(-t(H_M(V)+n\cdot 1_{M\setminus \Omega})\big)$ as $n\to\infty$, and where $P_\Omega:L^2(M)\to L^2(\Omega)$ denotes the natural projection.}\vspace{1mm}

(B) Assume $M$ is a geodesically complete Riemannian spin manifold. \emph{For every spin bundle $\ISS\to M$ with corresponding Dirac operator $D$ acting on sections of $\ISS\to M$, and every spinor $\Psi \in \Gamma_{L^2}(M,\ISS)$ with $D\Psi \in \Gamma_{L^2}(M,\ISS)$ and $\Psi=0$ almost everywhere in $M\setminus \Omega$, one has the implication}
$$
\text{\emph{$D^2\Psi=0$  in $\Omega$  $\Rightarrow$ $D\Psi=0$ in $\Omega$ .}}
$$
(C) \emph{The first exit time of Brownian motion from $\Omega$ equals the first penetration time of Brownian motion to $M\setminus\Omega$}. This property in particular implies\footnote{We will not assume $M$ to be stochastically complete} that with $\mathbb{P}^x$ the Riemannian Brownian motion measure with initial point $x$ and $\IX$ the coordinate process on the path space of $M$ one has
\begin{align*} 
& \{ \IX_s\in \Omega\text{\emph{ for a.e. $s\in [0,t]$ and $\IX_s\in M$ for all $s\in [0,t]$}}\}\\
&=_{\mathbb{P}^x}  \{ \IX_s\in \Omega\text{ \emph{for all $s\in [0,t]$}}\}\quad\text{\emph{ for all $x\in \Omega$, $t>0$}}.
\end{align*}
To the best of our knowledge, property (C) has been introduced by D. Stroock \cite{stroock} in the Euclidean space $\IR^m$, and is usually referred to as the \emph{Kac regularity} of $\Omega$.\vspace{2mm}


While (A), (B) and (C) seem unrelated at first glance, the main results of this paper show that the above three problems are much more correlated than one might expect: Indeed we prove that given an arbitrary open subset $\Omega\subset M$,
\begin{itemize}
\item[(i)] (A) is equivalent to 
$$
\quad\quad\quad\text{(A')}\>\>\text{ for all $f\in W^{1,2}_{0}(M;V)$ with $f=0$ a.e. in $M\setminus \Omega$ one has $f|_\Omega\in W^{1,2}_{0}(\Omega;V)$}
$$
for every $V$ as in (A); here $W^{1,2}_{0}(M;V)$ and $W^{1,2}_{0}(\Omega;V)$ denote, respectively, the form domain of $H_M(V)$ and $H_{\Omega}(V)$,
\item[(ii)] (C) is equivalent to (A),
\item[(iii)] (C) implies (B), 
\item[(iv)] If $ \Omega$ is locally Lipschitz regular (cf. Definition \ref{lipreg}), then one has (C).
\end{itemize}

While (i), (ii) and a variant of (iv) have been obtained by I.W. Herbst and Z.X. Zhao in the Euclidean space $\IR^m$ \cite{herbst} for $V=0$ (cf. Remark \ref{ende}), our analysis is the first systematic treatment of these questions on manifolds, allowing in addition potentials and in fact covariant Schr\"odinger operators. 
As we only require a very weak regularity on the potential $V$, we can also treat Schr\"odinger operators that appear naturally in quantum mechanics, having potentials with Coulomb type singularities. In fact, we are going to treat these problems within the much more general class of covariant Schr\"odinger operators that act on sections of metric vector bundles over $M$, allowing to treat magnetic fields or squares of geometric Dirac operators simultaneously, making contact with (B). Our proofs rely on various (partially new) covariant Feynman-Kac formulas.\vspace{2mm}

{\bf Acknowledgements:} The authors would like to thank Stefano Pigola, Olaf Post, Peter Stollmann and an anonymous referee for very helpful remarks that have considerably improved the paper.

\section{Main result}

In the sequel, we work in the smooth category, that is, all differential, topological and geometric data (like manifolds, bundles, metrics and covariant derivatives) are understood to be smooth. In addition, any manifold is understood to be without boundary, unless otherwise stated. Without loss of generality, we will consider only complex function spaces. Furthermore, for every manifold $M$ and every complex vector bundle $E\to M$ its fibers will be denoted with $E_x$, $x\in M$, and the corresponding space of sections of $E\to M$ having a certain regularity $\mathscr{C}$ like $\mathscr{C}=C^{\infty}_c$ etc. will be denoted with $\Gamma_{\mathscr{C}}(M,E)$. The smooth linear partial differential operator
$$
d_{\bullet}:\Gamma_{C^{\infty}}(M,\wedge^{\bullet} T^*_{\IC}M)\longrightarrow \Gamma_{C^{\infty}}(M,\wedge^{\bullet+1} T^*_{\IC}M)
$$ 
from $\wedge^{\bullet} T^*_{\IC}M\to M$ to itself denotes the exterior derivative acting on complexified differential forms.
\vspace{1mm}

Let $M$ be an arbitrary connected Riemannian $m$-manifold. As such, $M$ is equipped with its Riemannian volume measure $\mu$. Given a complex metric vector bundle $\IEE\to M$, the scalar product on the complex Hilbert space $\Gamma_{L^2}(M,\IEE)$ of Borel equivalence classes of square integrable sections of $\IEE\to M$ is simply denoted with 
$$
\left\langle f_1,f_2\right\rangle =\int (f_1,f_2) \Id\mu,
$$
with
$$
\left\|f\right\|^2=\int |f|^2 \Id\mu
$$
the induced norm. Given two such bundles $\IEE_1\to M$, $\IEE_2\to M$, the formal adjoint of a smooth linear partial differential operator $D$ mapping sections of $\IEE_1\to M$ to sections of  $\IEE_2\to M$ with respect to $\left\langle \cdot,\cdot\right\rangle $ is simply denoted with $D^{\dagger}$. The reader may find the basics of linear partial differential operators on vector bundles over Riemannian manifolds in \cite{gunbook}. The following conventions will be very convenient in the sequel:

\begin{Notation} If $\Omega\subset M$ is an open subset and $\IEE \to M$ is a complex vector bundle, we define
\begin{align*}
\Gamma_{C^{\infty}_c}(\Omega,\IEE):=\left\{f\in \Gamma_{C^{\infty}_c}(M,\IEE):\text{ $f$ is compactly supported in $\Omega$} \right\}\subset \Gamma_{C^{\infty}_c}(M,\IEE),
\end{align*}
and if $\IEE \to M$ is equipped with a metric, then
\begin{align*}
\Gamma_{L^2}(\Omega,\IEE):=\left\{f\in \Gamma_{L^2}(M,\IEE):\text{ $f=0$ $\mu$-a.e. in $M\setminus \Omega$} \right\}\subset \Gamma_{L^2}(M,\IEE).
\end{align*}
\end{Notation}

Then $\Gamma_{L^2}(\Omega,\IEE)$ is a closed subspace of $\Gamma_{L^2}(M,\IEE)$, thus a Hilbert space in itself. We denote with 
$$
P_{\Omega}: \Gamma_{L^2}(M,\IEE )\longrightarrow \Gamma_{L^2}(\Omega,\IEE )
$$
the orthogonal projection onto $\Gamma_{L^2}(\Omega,\IEE  )$. Note that $P_{\Omega}$ is nothing but the restriction map $f\mapsto f|_\Omega$. \\
Let $(t,y)\mapsto p_{\Omega}(t,x,y)$ denote the pointwise minimal solution of the heat equation on the open subset $\Omega\subset M$ with $\lim_{t\to 0+}p_{\Omega}(t,x,\cdot)=\delta_x$. We recall that the Kato class $\mathcal{K}(M)$ of $M$ is defined by all Borel functions $w$ on $M$ such that
\begin{align}\label{dre}
\lim_{t\to 0+}\sup_{x\in M}\int p_M(t,x,y) |w(y)| d\mu(y)=0.
\end{align}

Based on the results from \cite{gunbook} (cf. Chapter VII), we propose:

\begin{Definition}\label{popo} By a \emph{Kato-Schr\"odinger bundle over $M$}, we will understand a datum 
$$
(\IEE ,\nabla, V)\longrightarrow M
$$
with
\begin{itemize}
\item $\IEE\to M$ a complex metric vector bundle with finite rank; \item $\nabla$ a metric covariant derivative on $\IEE\to M$; 
\item $V:M\to \mathrm{End}(\IEE)$ is a Kato decomposable potential, that is, there exist pointwise self-adjoint Borel sections $V_{\pm}$ of $\mathrm{End}(E)\to M$ with $V_{\pm}(x)\geq 0$ for all $x\in M$ and
\begin{align}\label{pqyss}
 V=V_+-V_-, \quad  |V_+|\in L^1_{\loc}(M) , \quad |V_-|\in \mathcal{K}(M),
\end{align}
where $|\cdot|$ denotes the fiberwise operator norm.
\end{itemize}

\end{Definition}

It follows in the above situation that $|V_-|\in L^1_{\loc}(M)$ (cf. Lemma VI.3 in \cite{gunbook}) and so $|V|\in L^1_{\loc}(M)$. In view of \cite{gri}
$$
\int p_M(t,x,y) d\mu (y) \leq 1\quad \text{ for all $(t,x)\in (0,\infty)\times M$, }
$$
it follows that every bounded Borel function $w$ on $M$ satisfies (\ref{dre}), and the reader may find various (weighted) $L^q$-assumptions on $w$ that imply (\ref{dre}) in Chapter VI from \cite{gunbook}. In particular, every pointwise self-adjoint $L^1_{\mathrm{loc}}$-section $V$ of $\mathrm{End}(E)\to M$ which is bounded from below, in the sense that for some constant $C\in \IR$ one has $V(x)\geq C$ for all $x\in M$, satisfies (\ref{pqyss}). As $L^1_{\loc}(M)$ and $ \mathcal{K}(M)$ are linear spaces, the sum of two Kato decomposable potentials is also Kato decomposable, which is seen by decomposing the sum $V+W$ in the form 
$$
V+W=(V_++W_+)-(V_-+W_-).
$$ 
In particular, for $V$ Kato decomposable, $\lambda\in [0,\infty)$ and $A\subset M$ a Borel set, the potential $V+\lambda 1_A$, where $1_A$ acts as a scalar, is Kato decomposable.

\begin{Notation} In the situation of Definition \ref{popo}, let $\Omega\subset M$ be an open subset. Then $\Gamma_{W^{1,2}_0}(\Omega,\IEE;\nabla,V)$ is defined to be the space of all $f\in \Gamma_{L^2}(\Omega,\IEE;\nabla,V)$ which admit a sequence $(f_n)\subset \Gamma_{C^{\infty}_c}(\Omega,\IEE;\nabla,V)$ such that
\begin{align*}
&\left\|f_n-f\right\|\to 0\quad\text{as $n\to \infty$},\\
&\int\big(\nabla^{\dagger}\nabla (f_n-f_m),(f_n-f_m) \big)d\mu+\int \big(V(f_n-f_m),(f_n-f_m)\big) d\mu \to 0\quad\text{as $n,m\to \infty$}.
\end{align*}
By Corollary XIII.2 in \cite{gunbook} the symmetric densely defined sesquilinear form
$$
\Gamma_{C^{\infty}_c}(\Omega,\IEE)\times \Gamma_{C^{\infty}_c}(\Omega,\IEE)\ni (f_1,f_2)\longmapsto  \int( \nabla^{\dagger}\nabla f_1,f_2 )d\mu+\int  (Vf_1,f_2 ) d\mu \in \IC
$$
is closable and semibounded from below in $\Gamma_{L^2}(\Omega,\IEE)$. Its closure $\langle \cdot, \cdot \rangle_{\nabla,V,*}$ has the domain of definition $\Gamma_{W^{1,2}_0}(\Omega,\IEE;\nabla,V)$ and is explicitly given as follows: For $f, \tilde{f} \in  \Gamma_{W^{1,2}_0}(\Omega,\IEE;\nabla,V)$ with defining sequences $(f_n), (\tilde{f}_n)\subset \Gamma_{C^{\infty}_c}(\Omega,\IEE;\nabla,V)$, one has
$$
 \langle f, \tilde{f}  \rangle_{\nabla,V,*}=\lim_{n\to \infty} \left( \int( \nabla^{\dagger}\nabla f_n,\tilde{f}_n )d\mu +\int  (Vf_n,\tilde{f}_n ) d\mu\right), 
$$
and this number does not depend on the defining sequences. The semibounded from below and self-adjoint operator in $\Gamma_{L^2}(\Omega,\IEE)$ induced by $\langle \cdot, \cdot  \rangle_{\nabla,V,*}$ is denoted with $H_\Omega(\nabla,V)$. \\
In addition, $\Gamma_{W^{1,2}_0}(\Omega,\IEE;\nabla,V)$ becomes a Hilbert space with respect to the scalar product 
\begin{align*}
 \langle f_1,f_2  \rangle_{\nabla,V}:= (1+C) 
 \langle f_1,f_2 \rangle + \left\langle f_1,f_2 \right\rangle_{\nabla,V,*},
\end{align*}
where $C\geq 0$ is any constant satisfying 
\begin{align}\label{swq}
\int( \nabla^{\dagger}\nabla f, f ) d\mu+ \int (V f, f)d\mu \geq -C\left\|f\right\|^2\quad\text{ for all $ f\in \Gamma_{C^{\infty}_c}(M,\IEE)$},
\end{align}
noting that any two such constants produce equivalent scalar products.\\
Note also that $H_\Omega(\nabla,V)$ is nothing but the Friedrichs realization of $\nabla^{\dagger}\nabla +V$ in $\Gamma_{L^2}(\Omega,\IEE)$. Traditionally, if $\Omega\ne M$, one says that $H_\Omega(\nabla,V)$ is the self-adjoint realization of $\nabla^{\dagger}\nabla +V$ in $\Gamma_{L^2}(\Omega,\IEE)$ subject to Dirichlet boundary conditions. 
\end{Notation}

To simplify the notation in some special cases, we add:

\begin{Remark}\label{bemee} If $\IEE\to M$ is the trivial complex line bundle and $\nabla$ is the covariant derivative which is induced by the exterior differential, then we simply omit $\nabla$ and $\IEE$ everywhere in the notation. Note that in this case $\nabla^{\dagger}\nabla=:-\Delta$ is by definition the usual Laplace-Beltrami operator on functions. In case $V=0$ we will in addition omit $V$ everywhere in the notation. These conventions lead to natural notations such as $W^{1,2}_0(\Omega;V)$, $H_{\Omega}(V)$, $W^{1,2}_0(\Omega)$, $L^2(\Omega)$, $H_{\Omega}$. In particular, the self-adjoint operator $H_{\Omega}$ in $L^2(\Omega)$ is $-\Delta$ with Dirichlet boundary conditions and the integral kernel $\exp(-tH_\Omega)(x,y)$ of $\exp(-tH_{\Omega})$ is precisely $p_{\Omega}(t,x,y)$ \cite{gri}.
\end{Remark}

Here is an example from quantum mechanics:

\begin{Example} If $M$ is the Euclidean space $\IR^3$ and $\eta=\sum^3_{j=1}\eta_j dx^j$ is a smooth real-valued $1$-form on $\IR^3$ (considered as a magnetic potential), $Z\in\IN$, and $V:\IR^3\to  \IR$ is the Coulomb potential $V(x):=-Z/|x-x_0|$ with mass at $x_0\in\IR^3$, then with 
$$
\nabla^{\eta}: C^{\infty}(\IR^3)\longrightarrow \Gamma_{C^{\infty}}(\IR^3,T^*_{\IC}\IR^3),\quad
\nabla^{\eta}f:=\sum^3_{j=1}(\partial_j f+\sqrt{-1}f\eta_j) dx^j ,
$$
the datum 
$$
(\IR^3\times \IC, \nabla^{\eta}, V)\longrightarrow \IR^3
$$
is a Kato-Schr\"odinger bundle (cf. Proposition VI.14 in \cite{gunbook}). In this case, identifying sections in $\IR^3\times \IC\to \IR^3$ with functions on $M$, the semibounded self-adjoint operator $H_{\IR^3}(\nabla^{\eta},V)$ in $L^2(\IR^3)$ is the Hamilton operator of an atom having one electron and a nucleus having $Z$ protons, in the magnetic field given by 
$$
d\eta=\sum_{i<j} (\partial_i \eta_j-\partial_j \eta_i)dx^i \wedge dx^j.
$$
More generally, one could replace in this example $\IR^3$ with a Riemannian $3$-fold $M$ such that $p_M(t,x,y)$ satisfies a Gaussian upper bound of the form
$$
p_M(t,x,y)\leq C_1t^{-3/2}\exp\left(-\frac{d(x,y)^2}{C_2t}\right),\quad (t,x,y)\in (0,\infty)\times M\times M,
$$
taking 
$$
V(x):=-Z\int^{\infty}_0 p_M(t,x,x_0)d t
$$
to be the Coulomb potential with mass at $x_0$. One could even take the electron\rq{}s spin into account, as is shown in \cite{guneysu2}.
\end{Example}

Here is an example from geometry:

\begin{Example}\label{diracc} 1. A \emph{geometric Dirac bundle over $M$} is understood to be a datum 
$$
(\IEE;c,\nabla)\longrightarrow M
$$
  such that
\begin{itemize}
\item $\IEE\to M$ is a complex metric vector bundle;
\item  $c$ is a Clifford multiplication on $\IEE\to M$, in the sense that $c$ is a homomorphism of real (!) vector bundles $ c:  TM\to\mathrm{End}(\IEE)$, such that for all $X\in \Gamma_{C^{\infty}}(M,TM)$ one has $
c(X)=-c(X)^*$ and $c(X)^{*}c(X)=\left|X\right|^2$;
\item $\nabla$ is a Clifford connection on $(\IEE;c)\to M$, that is, $\nabla$ is a metric covariant derivative on $\IEE\to M$ such that for all $X,Y\in \Gamma_{C^{\infty}}(M,TM)$, $\Psi\in\Gamma_{C^{\infty}}(M,\IEE)$ one has
\[
 \nabla_X(c(Y)\Psi)=c(\nabla^{T M}_X Y)\Psi+ c(Y)\nabla_X \Psi. 
\]
\end{itemize}
Then the associated geometric Dirac operator is the first order linear partial differential operator on $\IEE\to M$ defined by $D(c,\nabla):=\sum^m_{j=1} c(e_j)\nabla_{e_j}$, where $e_j$ is any local orthonormal-frame for $TM\to M$ and $m=\dim M$. The operator $D(c,\nabla)$ is formally self-adjoint with symbol $c$, in particular elliptic. By the abstract Lichnerowicz formula \cite{nico,lawson},
$$
V(c,\nabla):= D(c,\nabla)^2-\nabla^{\dagger}\nabla:M\to \mathrm{End}(\IEE)
$$
is a smooth potential. In case $V(c,\nabla)$ is Kato decomposable (which is the case, e.g., if $V(c,\nabla)$ is bounded from below), we call
$$
(\IEE;c,\nabla)\longrightarrow M
$$
a \emph{geometric Kato-Dirac bundle}, and we get the Kato-Schr\"odinger bundle
$$
\big(\IEE, \nabla,V(c,\nabla) \big)\longrightarrow M.
$$
In this case, $\Gamma_{W^{1,2}_0}(\Omega,\IEE;\nabla,V(c,\nabla))$ is given by all $f\in \Gamma_{L^2}(\Omega,\IEE)$ with $D(c,\nabla)f\in \Gamma_{L^2}(\Omega,\IEE)$ in the sense of distributions, which admit a sequence $f_n\in \Gamma_{C^{\infty}_c}(\Omega,\IEE)$, $n\in\IN$, with
$$
\left\|f_n-f\right\|+\left\|D(c,\nabla)f_n-D(c,\nabla)f\right\|\to 0.
$$
We have
$$
\left\langle f_1,f_2 \right\rangle_{ \nabla,V(c,\nabla),*}=  \left\langle D(c,\nabla)f_1,D(c,\nabla)f_2 \right\rangle ,
$$
so that in particular the scalar product on $\Gamma_{W^{1,2}_0}(\Omega,\IEE;\nabla,V(c,\nabla))$ is given by 
$$
\left\langle f_1,f_2 \right\rangle_{ \nabla,V(c,\nabla)}= \left\langle f_1,f_2\right\rangle+ \left\langle D(c,\nabla)f_1,D(c,\nabla)f_2 \right\rangle .
$$
2. For example, if $M$ is a spin manifold, then every spin structure on $M$ canonically induces a geometric Dirac bundle over $M$ \cite{lawson,nico}, which is a geometric Kato-Dirac bundle, if the scalar curvature of $M$ is Kato decomposable.\\
3. Another example of a geometric Dirac bundle over $M$ (which does not require any topological assumptions on $M$) is given by taking $\IEE=\wedge T^*_{\IC}M$ with the metric induced by the Riemannian metric on $M$ and $\nabla$ the covariant derivative which is induced by the Levi-Civita connection on $M$, and 
$$
c(\alpha)\beta:=\alpha\wedge\beta - \iota_{\alpha^{\sharp}}\beta, \quad \alpha\in \Gamma_{C^{\infty}}(M,T^*_{\IC}M),\beta\in \Gamma_{C^{\infty}}(M,\wedge T^*_{\IC}M),
$$
where $\iota_{\alpha^{\sharp}}$ denotes the contraction with the vector field $\alpha^{\sharp}$ which is induced by $\alpha$ via the Riemannian metric on $M$. In this case, one can calculate that $D(c,\nabla)=d+d^{\dagger}$, so that $D(c,\nabla)^2$ is by definition the Laplace-Beltrami operator acting on differential forms. Note that the restriction of $D(c,\nabla)^2$ to $0$-forms is precisely the usual Laplace-Beltrami operator $-\Delta$ on functions, and that the restriction of $D(c,\nabla)$ to $0$-forms can be identified with the gradient (cf. Proposition 11.2.1 in \cite{nico} for detailed proofs of these facts). 
\end{Example}

Given a Kato-Schr\"odinger bundle
$$
(\IEE ,\nabla, V)\longrightarrow M,
$$
the inclusion
$$
\Gamma_{W^{1,2}_0}(\Omega,\IEE ;\nabla ,V  )\subset\big\{  f\in \Gamma_{W^{1,2}_0}(M,\IEE;\nabla,V ): f|_{M\setminus \Omega}=0\>\>\text{\rm $\mu$-a.e. }\big\}
$$
is easily seen to be satisfied without any further assumptions on the underlying data: Indeed, given $\psi\in\Gamma_{W^{1,2}_0}(\Omega,\IEE ,\nabla , V  )$ we can per definitionem pick a sequence $\psi_{n}\in\Gamma_{C^{\infty}_c}(\Omega,\IEE )$ with 
$\left\|\psi_n- \psi\right\|_{  \nabla ,V }\to 0$, so that $\left\|\psi_n- \psi\right\|\to 0$ and so $\psi=0$ $\mu$-a.e. in $M\setminus \Omega$.\vspace{2mm}

\emph{The question we address in this paper is: Under which condition on $\Omega$ does the reverse inclusion
\begin{align}\label{poa44}
\Gamma_{W^{1,2}_0}(\Omega,\IEE ;\nabla,V  )\supset\big\{  f\in \Gamma_{W^{1,2}_0}(M,\IEE;\nabla,V ):  f|_{M\setminus \Omega}=0\>\>\text{\rm $\mu$-a.e. }\big\}
\end{align}
hold true?}
\vspace{2mm}

By the above, (\ref{poa44}) is equivalent to

\begin{align}\label{poa}
\Gamma_{W^{1,2}_0}(\Omega,\IEE ;\nabla,V  )=\big\{  f\in \Gamma_{W^{1,2}_0}(M,\IEE;\nabla,V ):  f|_{M\setminus \Omega}=0\>\>\text{\rm $\mu$-a.e. }\big\}.
\end{align}

The validity of (\ref{poa}) turns out to be equivalent to the corresponding Dirichlet heat-flow of $H_{\Omega}(\nabla,V)$ on $\Omega$ being realized as the heat-flow of $H_M(\nabla,V)$ on $M$ perturbed by the \lq{}potential\rq{} $\infty\cdot 1_{M\setminus \Omega}$, for one has:

\begin{Proposition}\label{main1} Let $\Omega\subset M$ be an arbitrary open subset and let $(\IEE,\nabla,V)\to M$ be a Kato-Schr\"odinger bundle over $M$. Then one has (\ref{poa}), if and only if for all $t\geq 0$ one has 
\begin{align}\label{aposssay}
\lim_{n\to\infty} \mathrm{exp}\big(-tH_M(\nabla,V+n1_{M\setminus \Omega})\big)  =  \mathrm{exp}\big(-t H_\Omega(\nabla, V)\big)P_{\Omega}
\end{align}
strongly as bounded operators in $\Gamma_{L^2}(M,\IEE)$.
\end{Proposition}

The proof of Proposition \ref{main1} will be given in Section \ref{pon}. It is a consequence of monotone convergence results for sesquilinear forms. \\
As one might guess, the validity of (\ref{poa}) or (\ref{aposssay}) should depend on the local regularity of $\partial \Omega$. In order to make precise how (\ref{poa}) or (\ref{aposssay}) depend on the local regularity of $\partial \Omega$, the probabilistic Definition \ref{paa} below will turn out to be crucial. To this end, let $W(\hat{M})$ denote the Wiener space of continuous paths $\omega:[0,\infty)\to \hat{M}$ with its Borel-sigma-algebra $\mathcal{F}$, where $ \hat{M}=M\cup\{\infty_M\}$ is defined to be the (essentially uniquely determined) Alexandroff-compactification of $M$ in case $M$ is noncompact, and $\hat{M}:=M$ if $M$ is compact.\footnote{Note that every open subset of $M$ is also open in $\hat{M}$. In addition, the closure of a subset $\Omega$ of $M$ is always understood with respect to the topology of $M$ and not the topology of $\hat{M}$, noting that both closures coincide if and only if the closure of $\Omega$ in $M$ is a compact subset of $M$.} In any case, $W(\hat{M})$ is equipped with the topology of locally uniform convergence. Let 
$$
\IX: [0,\infty)\times W(\hat{M})\longrightarrow \hat{M},\quad \IX_t(\omega):= \omega(t)
$$
denote the coordinate process, and for every $x\in M$, the symbol $\mathbb{P}^x $ stands for the Riemannian Brownian motion measure on $\mathcal{F}$ with
$$
\mathbb{P}^x\{\IX_0=x\}=1.
$$
In other words, the transition density of $\mathbb{P}^x$ with respect to $\mu$ is given by the natural extension of $ \exp(-t H_M)(x,y)$ to $(0,\infty)\times \hat{M}\times \hat{M}$. We refer the reader to Section 8 in \cite{grio} and the references therein for the definition and the existence of $\mathbb{P}^x$.\\
With $\mathcal{F}_*$ the filtration of $\mathcal{F}$ which is generated by $\IX$, the family $\mathbb{P}^{x}$, $x\in M$, has the strong Markov property for stopping times in $\mathcal{F}_*$. If $\Omega\subset  \hat{M}$ is an open set, then\footnote{In the sequel the infimum of an empty set is understood to be $\infty$.} 
\begin{align*}
\alpha_\Omega:W(\hat{M})\longrightarrow [0,\infty],\quad \alpha_\Omega:= \inf\big\{t>0: \IX_t\in  \hat{M}\setminus \Omega\big\}
\end{align*}
denotes the first exit time of $\IX$ from $\Omega$. It is well-known that $\alpha_\Omega$ is an $\mathcal{F}_*$-stopping time. In particular, if $M$ is noncompact, then the stopping time $\alpha_M$ is the explosion time of $\mathbb{X}$, and in this case the point at infinity $\infty_M$ of $\hat{M}$ is $\mathbb{P}^x$-almost surely (a.s.) absorbing for all $x\in M$, in the sense that
$$
\mathbb{P}^x\left(\{\alpha_M=\infty\}\bigcup \{\alpha_M<\infty\>\text{ and $\mathbb{X}_t=\infty_M$  for all $t\in [\alpha_M,\infty)$} \}\right)=1\quad\text{ for all $x\in M$}.
$$

Likewise, 
\begin{align*} \beta_\Omega:= \inf\big\{t>0: \int^t_0 1_{ \hat{M}\setminus \Omega}(\IX_s)  \Id s >0\big\}:W(\hat{M})\longrightarrow [0,\infty]
\end{align*}
denotes the first penetration time of $\IX$ into $ M\setminus \Omega$. One trivially has 
\begin{align}\label{didaa}
\alpha_\Omega\leq \beta_\Omega,
\end{align}
and again $\beta_\Omega$ induces an $\mathcal{F}_*$-stopping time. Note that both $\alpha$ and $\beta$ are pathwise monotonely increasing in $\Omega$, and that one pathwise has
$$
\beta_\Omega= \inf\big\{t\geq 0: \int^t_0 1_{ \hat{M}\setminus \Omega}(\IX_s)  \Id s >0\big\}= \inf\Big\{t\geq 0: |\{s\in [0,t]:\IX_s\notin \Omega\}|>0\Big\},
$$
with $|\{\dots\}|$ the Lebesgue measure of $\{\dots\}$.  If $\omega\in W(\hat{M})$ is such that $\omega(0)\in \Omega$, then one also gets
$$
\beta_\Omega(\omega) \geq \alpha_\Omega(\omega)= \inf\big\{t\geq 0: \omega(t)\in  \hat{M}\setminus \Omega\big\}>0.
$$

The following regularity result will be very useful in the sequel:

\begin{Lemma}\label{dnaarrr} Let $\omega\in W(\hat{M})$ be such that in case $\alpha_M(\omega)<\infty$ one has $\omega(t)=\infty_M$ for all $t\in [\alpha_M(\omega),\infty)$, let $\Omega$ be an open subset of $M$, and let $\Omega_n\subset \Omega$, $n\in\IN$, be open with $\Omega_n\subset \Omega_{n+1}$ for all $n\in\IN$, $\omega(0)\in \Omega_1$ and $\bigcup_{n\in\IN}\Omega_n=\Omega$.\\
\emph{a)} One has $\alpha_{\Omega_n}(\omega)\nearrow\alpha_{\Omega }(\omega) $ as $n \to\infty$.\\
\emph{b)} Assume in addition that for each $n\in \mathbb{N}$ there exists an open subset $\Upsilon_n\subset M$ such that $\Upsilon_n\cap \Omega=\Omega_n$ and $\overline{\Omega}=\bigcup_{n\in\IN}(\overline{\Omega}\cap \Upsilon_n)$. Then one has $\beta_{\Omega_n}(\omega)\nearrow\beta_{\Omega }(\omega) $ as $n \to\infty$.
\end{Lemma}

While Lemma \ref{dnaarrr} a) is well-known, a proof of Lemma \ref{dnaarrr} b) will be given in Section \ref{weqyx}.

\begin{Remark}\label{lop} If $\{\Upsilon_n:n\in\IN\}$ is a family of open subsets of $M$ such that $\Upsilon_n\subset \Upsilon_{n+1}$ for all $n\in\IN$ and $\bigcup_{n\in\IN}\Upsilon_n=M$, then the family $\Omega_n:=\Omega\cap \Upsilon_n$, $n\in\IN$, satisfies the hypothesis of Lemma \ref{dnaarrr} part b).
\end{Remark}

The following definition is motivated by Herbst/Zhao \cite{herbst} and Stroock \cite{stroock}, who treat the Euclidean space $\IR^m$:

\begin{Definition}\label{paa} An open set $\Omega\subset M$ is called \emph{Kac regular}, if one has
\begin{align}\label{apossaq}
\mathbb{P}^x \{ \alpha_\Omega =\min(\beta_\Omega,\alpha_M) \}=1\quad\text{ for all $x\in \Omega$}.
\end{align}
\end{Definition}

Note that we do not assume $M$ to be stochastically complete, that is $\alpha_M=\infty$ $\mathbb{P}^x$-a.s. for all $x\in M$, noting that in the latter case one has
$$
\mathbb{P}^x \{ \alpha_\Omega =\min(\beta_\Omega,\alpha_M) \}=\mathbb{P}^x \{ \alpha_\Omega =\beta_\Omega \},
$$
and some technical problems are not present, that we will have to deal with. The connection between problems such as (\ref{poa}) and Kac regularity is clarified in the following result, which has been established in the Euclidean space $\IR^m$ by Herbst/Zhao \cite{herbst}:

\begin{Proposition}\label{scal1} Let $\Omega\subset M$ be an arbitrary open subset. The following properties are equivalent:
\begin{itemize}
	\item $\Omega$ is Kac regular.  
	\item One has 
	\begin{align}\label{alkdd}
W^{1,2}_{0}(\Omega  )=\big\{f\in W^{1,2}_0 (M  ):  f|_{M\setminus \Omega}=0\>\>\text{\rm $\mu$-a.e. }\big\}.
\end{align}
\item For all $ t\geq 0$ one has 
\begin{align}\label{apossq}
\lim_{n\to\infty}\mathrm{exp}\big(-t H_M(   n 1_{M\setminus \Omega})\big)   =  \mathrm{exp}\big(-t H_\Omega  \big)P_\Omega
\end{align}
strongly as bounded operators in $L^2(M)$.
\end{itemize}
\end{Proposition} 

The proof of Proposition \ref{scal1} will be given in Section \ref{pon}. Note that in our notation $H_M(   n 1_{M\setminus \Omega})$ is the Friedrichs realization of $-\Delta+n 1_{M\setminus \Omega}$ in $L^2(M)$, and $H_\Omega $ is the Dirichlet realization of $-\Delta$ in $L^2(\Omega)$. The equivalence of (\ref{alkdd}) and (\ref{apossq}) follows immediately from specializing Proposition \ref{main1} to the scalar case. The equivalence of (\ref{apossq}) and Kac-regularity will be established by using a Feynman-Kac formula.

\begin{Remark}\label{scal2}1. The potential theoretic variant of (\ref{alkdd}) does not see anything from the geometry of $\Omega$, namely \cite{fuku}, for \emph{every open $\Omega\subset M$} one has
$$
W^{1,2}_0(\Omega  )=\big\{ f\in W^{1,2}_0(M ): \breve{f}|_{M\setminus \Omega}=0\>\> \text{\rm q.e.}\big\},
$$
where \lq{}q.e.\rq{} (quasi everywhere) is understood with respect to the capacity associated with the regular strongly local Dirichlet form 
$$
\left\langle f_1,f_2 \right\rangle_{ * }= \int (df_1,df_2) d\mu
$$
 in $L^2(M)$ with domain of definition $W^{1,2}_0(M )$, and where $\breve{f}$ denotes the quasi-continuous representative of $f$. This again reflects the subtlety of the questions under investigation.\\
2. It was conjectured in \cite{simon} by B. Simon that open subsets $\Omega\subset \IR^m$ of the Euclidean space with $\IR^m\setminus \Omega$ perfect satisfy (\ref{apossq}). A counterexample to this conjecture was given by L.I. Hedberg \cite{hedberg} and P. Stollmann \cite{stollmann1}  (cf. \cite{stollmann2}, p.127), who even show that there exists $\Omega\subset \IR^m$ open with $\IR^m\setminus \Omega$ compact and
$$
\IR^m\setminus \Omega=  \overline{\mathring{\IR^m\setminus \Omega }}
$$
such that (\ref{alkdd}) and thus (\ref{apossq}) fails. Note that this counterexample also entails that even for quasi-continuous $f\rq{}s$ the assumption $f|_{M\setminus \Omega}=0$ a.e. does not imply that $f|_{M\setminus \Omega}=0$ q.e. 
\end{Remark}

In order to formulate our main result, we add:

\begin{Definition}
\label{locallylip}
Let $\Omega\subset M$ be an open subset of $M$. We say that $\Omega$ has a \emph{locally Lipschitz boundary}, if for every $p\in \overline{\Omega}\setminus \Omega=:\partial \Omega$ there exists a pair of open neighborhoods $V$, $X$ of $p$, and a  smooth diffeomorphism $\alpha: X\rightarrow A\subset \mathbb{R}^m$ such that: 
\begin{enumerate}
\item $\overline{V}\subset X$,
\item $\overline{V}$ is compact,
\item $\alpha(V\cap \Omega)$ is a bounded Lipschitz subset of $\mathbb{R}^m$. 
\end{enumerate}
\end{Definition}
We refer for instance to \cite{daco} or to \cite{stein} for the definition of open bounded Lipschitz subset of $\mathbb{R}^m$.

\begin{Definition}\label{lipreg} An open subset $\Omega\subset M$ is called \emph{locally Lipschitz regular}, if there exists a family $\{\Omega_n: n\in\IN\}$ of open subsets of $\Omega$ having locally Lipschitz boundary and compact closure in $M$, such that for all $n\in \IN$ there exists an open subset $\Upsilon_n \subset M$ such that $\Upsilon_n\cap \Omega = \Omega_n$ and $\overline{\Omega}=\bigcup_{n\in\IN}(\overline{\Omega}\cap \Upsilon_n)$.
\end{Definition}

Note that the local Lipschitz regularity of an open subset and that an open subset has a locally Lipschitz boundary are both concepts which, by definition, do not depend on the Riemannian metric on $M$. This underlines the importance of part b) of our following main result:

\begin{Theorem}\label{main2} Let $\Omega$ be an arbitrary open subset of $M$.\\
{\rm a)} If $\Omega$ is Kac regular, then for every Kato-Schr\"odinger bundle $(\IEE,\nabla,V)\to M$ one has (\ref{poa}) and (\ref{aposssay}).\\
{\rm b)} If $\Omega$ is locally Lipschitz regular, then $\Omega$ is Kac regular.
\end{Theorem}

Using transversality theory in the smooth category one finds:

\begin{Corollary} Every open subset $\Omega\subset M$ having a smooth boundary is locally Lipschitz regular and thus Kac regular.
\end{Corollary}

\begin{proof} The proof of the fact that $\Omega$ is locally Lipschitz regular follows from Lemma 18 in \cite{pigo}. For the sake of completeness, we provide a sketch of proof: pick a family $\{W_n:n\in\IN\}$ of relatively compact open subsets of $M$ with smooth boundary such that $W_n\subset W_{n+1}$ for all $n\in\IN$ and $\bigcup_{n\in\IN}W_n=M$. Let $i_n:W_n\rightarrow M$ be the natural inclusion of $W_n$ in $M$. Using classical results of transversality theory in differential topology we can find for each $n\in \mathbb{N}$ a smooth embedding $j_n:\overline{W_n}\rightarrow M$ such that $j_n$ and $i_n$ are arbitrarily close in the Whitney strong topology  and $j_n(\partial\overline{W_n})$ is transverse to  $\partial\overline{\Omega}$. Altogether this tells us that $\{j_n(W_n): n\in\IN\}$ is a family of open subsets of $M$ such that $j_n(W_n)\subset j_{n+1}(W_{n+1})$ for all $n\in\IN$ and $\bigcup_{n\in\IN}j_n(W_n)=M$ in a way that $\Omega_n:=j_n(W_n)\cap \Omega$ has a locally Lipschitz boundary. It follows from Remark \ref{lop} that the set $\Omega$ becomes a locally Lipschitz regular set. We refer to \cite{GuPo} and \cite{MoHi} for classical results about transversality theory and for the definition of Whitney topology.
\end{proof}

Moreover we have the following proposition.

\begin{Proposition}
Let $\Omega\subset M$ be a locally Lipschitz regular open subset. Then $\Omega$ has locally Lipschitz boundary.
\end{Proposition}

\begin{proof}
Let $p\in \overline{\Omega}\setminus \Omega$ and let $m_0$ be such that $p\in \Upsilon_{m_0}$. Let $W$ be any open neighborhood of $p$ such that  $\overline{W}\subset \Upsilon_{m_0}$. Then it is easy to check that $W\cap (\overline{\Omega}\setminus \Omega)=W\cap (\overline{\Omega}_{m_0}\setminus \Omega_{m_0})$, where we recall that $\Upsilon_{m_0}\cap \Omega=\Omega_{m_0}$. Indeed if $q\in W\cap (\overline{\Omega}\setminus \Omega)$ then for every open neighborhood $U$ of $q$ with $U\subset W$  we have $U\cap \Omega\neq \emptyset$ and thus $U\cap \Omega_{m_0}=U\cap \Upsilon_{m_0}\cap \Omega\neq \emptyset$. Thus $q\in W\cap (\overline{\Omega}_{m_0}\setminus \Omega_{m_0})$.  On the other hand if $z\in W\cap (\overline{\Omega}_{m_0}\setminus \Omega_{m_0})$, by the fact that $W=(W\cap (\overline{\Omega}\setminus \Omega))\cup (W\cap \Omega)$ and $W\cap \Omega\subset \Omega_{m_0}$, we have necessarily that $z\in W\cap (\overline{\Omega}\setminus \Omega)$.  Now, as $\Omega_{m_0}$ has locally Lipschitz boundary, we can find a pair of open neighborhoods $V$, $X$ of $p$, and a  smooth diffeomorphism $\alpha: X\rightarrow A\subset \mathbb{R}^m$ such that $\overline{V}\subset X$, $\overline{V}$ is compact and $\alpha(V\cap \Omega_{m_0})$ is a bounded Lipschitz subset of $\mathbb{R}^m$. Moreover, without loss of generality, we can assume that $X\subset W$. Therefore, as $p$ is any point in $\overline{\Omega}\setminus \Omega$, this shows that $\Omega$ is an open subset of $M$ with locally Lipschitz boundary.
\end{proof}

We expect that also the other implication holds true; more precisely we {\bf conjecture} that every open subset of a manifold having a locally Lipschitz boundary is locally Lipschitz regular. The investigation of this property seems to require an appropriate transversality theory in the locally Lipschitz category.\vspace{3mm}

Theorem \ref{main2} a) is new even in the Euclidean case, where so far only $-\Delta$ has been treated, and not even Schr\"odinger operators $-\Delta+V$. The proof of Theorem \ref{main2} will be given in Section \ref{pon}. We continue with remarks on this proof: The proof of Theorem \ref{main2} a) relies on Proposition \ref{main1} and Proposition \ref{scal1}, while both worlds are linked through various partially new covariant Feynman-Kac formula. We believe it is a very surprising fact, that no condition on the negative part of $V$ is needed in Theorem \ref{main2} a), as such a condition is certainly required for the covariant Feynman-Kac formula to hold (roughly speaking, because Brownian paths have an infinite speed). The point here is that using rather subtle approximation arguments, one can reduce everything to the case of $V$'s that are bounded from below by a constant. \\
The proof of Theorem \ref{main2} b) relies again on Proposition \ref{scal1}: In a first step , we will give an analytic proof of the stronger statement
\begin{align}\label{asoiss}
W^{1,2}_0(\Omega  )=\big\{   f\in W^{1,2}(M  ): f|_{M\setminus \Omega}=0 \>\>\text{\rm $\mu$-a.e. } \big\},
\end{align}
under the assumption that $\Omega$ is relatively compact with $\partial \Omega$ locally Lipschitz, directly verifying (\ref{alkdd}). As usual, $W^{1,2}(M  )$ is defined to be the space of all $f\in L^2(M)$ such that $df\in\Gamma_{L^2}(M,T^*M)$ in the sense of distributions. In a second step we will then combine this local result with a probabilistic approximation argument, directly confirming (\ref{apossaq}).

\begin{Remark} \label{ende}
1. It is false that every open subset $\Omega\subset M$ with smooth boundary satisfies\footnote{But of course this statement is true by Proposition \ref{scal1} and Theorem \ref{main2} b), if $M$ is geodesically complete, as then $W^{1,2}(M )=W^{1,2}_0(M   )$ \cite{aubin}. } (\ref{asoiss}). As a counterexample, one can consider the Euclidean ball and its open upper half
$$
M:= \{(x,y,z)\in \IR^3: x^2+y^2+z^2<1\} ,\quad \Omega:=\{(x,y,z)\in M: z>0\},
$$
with the Euclidean metric.\\
2. It is shown in \cite{herbst} that in the Euclidean case $\Omega$'s with $\partial \Omega$ having the segment property are Kac regular. In fact, many comparable Euclidean results can be found in P. Stollmann\rq{}s diploma thesis \cite{stollmann1}. The segment property is more general then being locally Lipschitz, but is also a concept that does not apply to manifolds. 
\end{Remark}

In view of Example \ref{diracc}.3, the following result includes (when applied to $0$-forms) a criterion for a harmonic function $f$ on an open subset $\Omega\subset M$ with $f=0$ $\mu$-a.e. in $M\setminus \Omega$ to have vanishing gradient in $\Omega$ (cf. \cite{herbst} for the Euclidean case):

\begin{Corollary}\label{pw} Let $M$ be geodesically complete and let $(\IEE;c,\nabla)\to M$ be a geometric Kato-Dirac bundle. Then for every open subset $\Omega\subset M$ which is Kac regular, and every $\Psi$ with
\begin{align*}
&\Psi\in \Gamma_{L^2}(M, \IEE),\quad D (c,\nabla)\Psi\in \Gamma_{L^2}(M, \IEE ),\quad\text{$\Psi=0$ $\mu$-a.e. in $M\setminus \Omega$,}
\end{align*}
one has the implication
$$
D (c,\nabla)^2\Psi=0\>\>\>\text{in $\Omega$}\quad\Rightarrow  \quad\text{$D(c,\nabla)\Psi=0$ in $\Omega$},
$$
where above the action of $D (c,\nabla)^2$ and $D(c,\nabla)$ is understood in the distributional sense.
\end{Corollary}

\begin{proof} Recall that $(\IEE;c,\nabla)\to M$ induces the Kato-Schr\"odinger bundle
$$
(\IEE,\nabla, V(c,\nabla))\longrightarrow M,
$$
As $M$ is geodesically complete one has 
\begin{align}\label{helpme}
\Psi\in \Gamma_{W^{1,2}_0}\big(M, \IEE; \nabla,V(c,\nabla)\big),
\end{align}
which in view of the product rule for $D (c,\nabla)$ (p. 171 in \cite{gunbook}) and a manifold version of the Meyers-Serrin Theorem (Theorem 2.9 in \cite{guidetti}) can be seen as in the case of functions \cite{aubin}. Now (\ref{helpme}) and the Kac regularity of $\Omega$ imply
$$
\Psi \in\Gamma_{W^{1,2}_0}\big(\Omega, \IEE ,\nabla , V(c,\nabla) \big),
$$ 
and we can pick a sequence  $\Psi_n\in \Gamma_{C^{\infty}_c}(\Omega,\IEE )$ with 
$$
\lim_{n\to\infty}\|\Psi_n-\Psi  \|_{ \nabla ,V(c,\nabla)  }\>=0.
$$
It follows that 
\begin{align*}
&\big\| D(c,\nabla ) \Psi    \big\|^2 \>= \lim_{n\to\infty} \big\langle D (c,\nabla)     \Psi_n, D (c,\nabla)     \Psi     \big\rangle\\
&= \int_\Omega ( D (c,\nabla)\Psi_n,  D (c,\nabla)  \Psi    )\Id\mu =\int_\Omega ( \Psi_n,  D (c,\nabla)^2  \Psi    )\Id\mu =0,
\end{align*}
where the integration by parts is justified as $\Psi$ is smooth in $\Omega$ by local elliptic regularity, and $\Psi_n$ is smooth and compactly supported in $\Omega$.
\end{proof}

Note that in Corollary \ref{pw} the geodesic completeness of $M$ was only assumed to guarantee
$$
\Psi\in \Gamma_{W^{1,2}_0}(M, \IEE;\nabla, V(c,\nabla)),
$$
which one could assume instead.

\section{Proof of Lemma \ref{dnaarrr}}\label{weqyx}

\begin{proof}[Proof of Lemma \ref{dnaarrr}] a) This statement is well-known and elementary to check. \\
b) As this observation seems to be new, we have decided to add its proof (which is rather technical): Assume first that there exists $t_0\in (0,\infty)$ such that $\beta_\Omega(\omega) =t_0$. This means that the Lebesgue measure of the set $\{s\in [0,t_0]: \omega(s)\notin \Omega\}$ is zero, briefly 
$$
|\{s\in [0,t_0]: \omega(s)\notin \Omega\}|=0
$$
 and moreover that for each $\epsilon>0$ we have 
$$
|\{s\in [0,t_0+\epsilon]: \omega(s)\notin \Omega\}|>0.
$$
By the assumptions we can deduce the existence of a sufficiently big index $n_0>0$  such that $$\omega([0,t_0])\cap \Omega= \omega([0,t_0])\cap \Omega_{n_0}.$$ Indeed let $n_0>0$ be such that $\omega([0,t_0])\cap \overline{\Omega}\subset \bigcup_{i=1}^{n_0}(\overline{\Omega}\cap \Upsilon_i)$. Then $\omega([0,t_0])\cap \Omega\subset \bigcup_{i=1}^{n_0}(\Omega\cap \Upsilon_i)=\Omega_{n_0}$. In this way we have 
$$
|\{s\in [0,t_0]: \omega(s)\notin \Omega_n\}|=0
$$
 for each $n\geq n_0$, because 
$$
\{s\in [0,t_0]: \omega(s)\notin \Omega_n\}=\{s\in [0,t_0]: \omega(s)\notin \Omega_{n_0}\}=\{s\in [0,t_0]: \omega(s)\notin \Omega\}
$$
 and the latter has Lebesgue measure $0$. Moreover we have  
$$
|\{s\in [0,t_0+\delta]: \omega(s)\notin \Omega_n\}|>0
$$
 for each $\delta>0$ and for each $n\geq n_0$, because  
$$
\{s\in [0,t_0+\delta]: \omega(s)\notin \Omega\}\subset \{s\in [0,t_0+\delta]: \omega(s)\notin \Omega_n\}\subset \{s\in [0,t_0+\delta]: \omega(s)\notin \Omega_{n_0}\}
$$
 and 
$$
|\{s\in [0,t_0+\delta]: \omega(s)\notin \Omega\}|>0.
$$
 Therefore $\beta_{\Omega_n}(\omega)=t_0$ for each $n\geq n_0$ and hence we can conclude that $\lim \beta_{\Omega_n}(\omega)=\beta_\Omega(\omega)$ as $n\rightarrow \infty$. \\
Assume now  that there is no $t_0\in (0,\infty)$ such that $\beta_\Omega(\omega)=t_0$. Therefore 
$$
|\{s\in [0,\infty): \omega(s)\notin \Omega\}|=0.
$$
In this case we set $\beta_\Omega(\omega):=\infty$. Consider first the case where  $\overline{\omega([0,\infty))}$ is compact. Let 
$$
A:=\{s\in [0,\infty): \omega(s)\notin \Omega\}
$$
 and let $B:=[0,\infty)\setminus A$. As above we deduce the existence of an integer $n_0>0$  such that $\omega(B)\subset \Omega_{n_0}$. Since $\omega(B)\subset \Omega_{n_0}$ we have 
$$
\omega([0,\infty))\cap \Omega=\omega([0,\infty))\cap \Omega_n
$$
 for each $n\geq n_0$. Therefore
$$
|\{s\in [0,\infty): \omega(s)\notin \Omega_n\}|=0
$$
 for any $n\geq n_0$ and hence we can conclude that $\lim \beta_{\Omega_n}(\omega)=\beta_\Omega(\omega)$ as $n\rightarrow \infty$. Finally consider the case where $\overline{\omega([0,\infty))}$ is not compact. The sequence $\beta_{\Omega_n}(\omega)$ is increasing and unbounded, because given an arbitrarily big $t_0$ we can find an integer $n_0$ such that $\omega([0,t_0])\cap \Omega=\omega([0,t_0])\cap \Omega_{n_0}$. Hence $\beta_{\Omega_{n_0}}(\omega)\geq t_0$ and therefore $\lim \beta_{\Omega_n}(\omega)=\infty$ as $n\rightarrow \infty$.
 \end{proof}


\section{Proofs of main results}\label{pon}

%


We first introduce some notation that will be relevant in the sequel. Let 
$$
(\IEE,\nabla,V)\longrightarrow M
$$
be a Kato-Schr\"odinger bundle. 

\begin{Notation} Given an open subset $\Omega\subset M$ let
\begin{align}\label{aqya}
\Gamma_{\widetilde{W}^{1,2}_0}(\Omega,\IEE;\nabla,V):=\big\{  f\in \Gamma_{W^{1,2}_0}(M,\IEE ;\nabla,V):  f|_{M\setminus \Omega}=0\>\>\text{ \rm $\mu$-a.e. }\big\}\subset \Gamma_{W^{1,2}_0}(M,\IEE ;\nabla,V),
\end{align}
and let $\widetilde{H}_\Omega(\nabla,V)$ denote the self-adjoint nonnegative operator in $\Gamma_{L^2}(\Omega,\IEE )$
which corresponds to $\left\langle \cdot,\cdot\right\rangle_{\nabla,V,*}$ with domain of definition $\Gamma_{\widetilde{W}^{1,2}_0}(\Omega,\IEE;\nabla,V)$, a closed densely defined symmetric sesquilinear form in $\Gamma_{L^2}(\Omega,\IEE )$. Again we will follow the conventions from (\ref{bemee}) analogously.
\end{Notation}

The above auxiliary operator $\widetilde{H}_\Omega(\nabla,V)$ satisfies the following two important results, without any further assumptions on $\Omega$:

\begin{Proposition}\label{zbl} For every open subset $\Omega$ of $M$ and every $t\geq 0$ one has 
$$
\lim_{n\to\infty}\mathrm{exp}\big(-t H_M(\nabla,V+n1_{M\setminus \Omega})\big)=  \mathrm{exp}\big(-t \widetilde{H}_{\Omega}(\nabla,V)\big)P_{\Omega}
$$
strongly as bounded operators in $\Gamma_{L^2}(M,\IEE)$.
\end{Proposition}

\begin{proof} This is a simple application of Theorem \ref{monn}, a convergence result for quadratic forms: Regarding $n1_{M\setminus \Omega}$ as a bounded multiplication operator we set 
$$
Q_n:=\left\langle \cdot,\cdot\right\rangle_{\nabla,V+n1_{M\setminus \Omega},*}=\left\langle \cdot,\cdot\right\rangle_{\nabla,V,*}+ n\int_{M\setminus \Omega } (\cdot,\cdot) d\mu,\quad 
\dom(Q_n):= \Gamma_{ W^{1,2}_0 }(M,\IEE ;\nabla,V).
$$
Then with $S_{Q_n}$, $Q_{\infty}$, $S_{Q_\infty}$ given as in Theorem \ref{monn}, we can conclude as follows: Firstly, one has
$$
S_{Q_n}=H_M(\nabla,V+n1_{M\setminus \Omega}).
$$
It follows that $f\in \Gamma_{W^{1,2}_0 }(M,\IEE ;\nabla,V)$ is in $\dom(Q_{\infty})$ (in the notation of Theorem \ref{monn}), if and only if 
$$
\sup_{n\in\IN}n\int_{M\setminus \Omega} |f|^2 d\mu<\infty,
$$
which again is equivalent to $f=0$ in $M\setminus \Omega$ $\mu$-a.e., and for such $f$\rq{}s one has
$$
Q_{\infty}(f,f)=\lim_{n\to \infty}Q_n(f,f)=\left\langle f,f\right\rangle_{\nabla,V,*},
$$
proving $S_{Q_\infty}=\widetilde{H}_{\Omega}(\nabla,V)$ .
\end{proof}

Let us now prepare the formulation of various covariant Feynman-Kac formulae, that is, path integral representations of the semigroups corresponding to covariant Schr\"odinger operators. Note that the usual Feynman-Kac formula \cite{guneysu,hsu} is for semigroups of the form $\mathrm{exp}\big(-t H_M(\nabla,V)\big)$, and it is the aim of this section to derive formulae for the semigroups $\mathrm{exp}\big(-t \widetilde{H}_{\Omega}(\nabla,V) \big)$ and $\mathrm{exp}\big(-t H_{\Omega}(\nabla,V)\big)$. As these formulae rely on semimartingale theory and stochastic integrals (through the definition of the Stratonovich parallel transport), one needs to work here with a filtered probability, which allows to pick continuous versions of these semimartingales. We refer the reader to \cite{hsu,guneysu} for the basics of stochastic analysis on manifolds. To set the stage, suppose that for every $x\in M$ we are given an adapted Riemannian Brownian motion $\mathbb{X}(x)$ on $M$ with 
$$
\IP\{\mathbb{X}_0(x)=x\}=1,
$$ 
which is defined on some fixed filtered probability space which satisfies the usual assumptions of completeness and right-continuity. In other words, $\mathbb{X}(x)$ is an $M$-valued continuous and adapted process defined up to its explosion time, such that the law of $\IX(x)$ is $\mathbb{P}^x$. Given an open subset $\Omega\subset M$, let $\alpha_\Omega(x):=\alpha_\Omega(\IX(x))$ denote the first exit time of $\mathbb{X}(x)$ from $\Omega$, and let $\beta_\Omega(x):=\beta_\Omega(\IX(x))$ denote the penetration time of $\mathbb{X}(x)$ to $M\setminus \Omega$. Let $\pa^{\nabla}_t(x)$, $t<\alpha_M(x)$, be the Stratonovich parallel transport with respect to $\nabla$ along the paths of $\mathbb{X}(x)$, which, $\nabla$ being metric, is a random variable taking values in the unitary maps  $\IEE_{x} \to \IEE_{\IX(t)}$. So far $x\in M$ was arbitrary. On the other hand, for $\mu$-a.e. $x\in M$, we can define pathwise for $t<\alpha_M(x)$ a random variable $\mathscr{A}^{\nabla}_V(x,t)$ taking values in $\mathrm{End}( \IEE_x)$, to be the unique locally absolutely continuous solution of
\begin{align*}
&(\Id/\Id t)\mathscr{A}^{\nabla}_V(x,t)=-\mathscr{A}^{\nabla}_V(x,t)\big(\pa^{\nabla}_t(x)^{-1}V(\IX_t(x))\pa^{\nabla}_t(x)\big),\\
&\mathscr{A}^{\nabla}_V(x,0)=1_{\mathrm{End}( \IEE_{x})}.
\end{align*}

The existence and uniqueness of $\mathscr{A}^{\nabla}_V(x,t)$ is guaranteed by the $L^1_{\loc}$-assumption on $V$, which implies
$$
 \mathbb{P} \left\{|V(\IX(x))|\in L^1_{\loc}\big[0,\alpha_M(x)\big)\right\}=1\quad\text{ for $\mu$-a.e. $x\in M$,}
$$
and so
$$
\mathbb{P}\left\{\left|\pa^{\nabla}_t(x)^{-1}V(\IX_t(x))\pa^{\nabla}_t(x)\right|\in L^1_{\loc}\big[0,\alpha_M(x)\big)\right\}=1\quad\text{ for $\mu$-a.e. $x\in M$}.
$$
We refer the reader to \cite{guneysu} for detailed proofs of these facts.

\begin{Proposition}\label{aposssy} For every $\Omega\subset M$ open, $t\geq 0$, $f\in \Gamma_{L^2}(\Omega,\IEE )$, and $\mu$-a.e. $x\in M$ one has 
\begin{align}\label{feyn3}
\mathrm{exp}\big(-t \widetilde{H}_{\Omega}(\nabla,V) \big)f(x) =\int_{\{t<\min(\beta_\Omega(x), \alpha_M(x) )\}}  \mathscr{A}^{\nabla}_V(x,t)\pa^{\nabla}_t(x)^{-1}f(\IX_t(x)) \Id\mathbb{P}.  
\end{align}
\end{Proposition}

Note that the expectation in (\ref{feyn3}) is on the fiber $\IEE_x$, as $f(\IX_t(x))\in \IEE_{\IX_t(x)}$ in $\{t<\alpha_M(x)\}$. In the scalar case we get

\begin{align}\label{pqvr3}
\mathrm{exp}\big(-t \widetilde{H}_{\Omega}(w)\big)f(x)=\int_{\{t<\min(\beta_\Omega(x),\alpha_M(x) )\}}   \mathrm{exp}\big(-\int^t_0 w(\IX_s)ds\big)f(\IX_t)\Id\mathbb{P} .
\end{align}

\begin{proof}[Proof of Proposition \ref{aposssy}] We start by recalling the usual covariant Feynman-Kac formula \cite{guneysu} 
$$
\mathrm{exp}\big(-t H_M(\nabla,W)\big)f(x)=\int_{\{t< \alpha_M(x)\}}  \mathscr{A}^{\nabla}_W(x,t)\pa^{\nabla}_t(x)^{-1}f(\IX_t(x)) \Id\mathbb{P},
$$
which is valid, for example if the potential $W$ is in $L^2_{\loc}(M)$ is such that for some constant $A\in\IR$ and all $x\in M$, the eigenvalues of $W(x)\in\mathrm{End}(\IEE_x)$ are $\geq A$. In particular, this statement includes that 
\begin{align}\label{domii}
\int_{\{t< \alpha_M(x)\}} \left| \mathscr{A}^{\nabla}_W(x,t)\pa^{\nabla}_t(x)^{-1}f(\IX_t(x))\right| \Id\mathbb{P}<\infty.
\end{align}

Applying this with $W=V+n1_{M\setminus \Omega}$ and using Proposition \ref{zbl} we have, possibly after taking a subsequence of $$
\mathrm{exp}\big(-t H_M(\nabla,V+n1_{M\setminus \Omega})\big)f$$ 
if necessary, that for $\mu$-a.e. $x\in M$ it holds that
\begin{align*}
 &\mathrm{exp}\big(-t \widetilde{H}_{\Omega}(\nabla,V)\big)f(x)=\lim_{n\to\infty}\mathrm{exp}\big(-t H_M(\nabla,V+n1_{M\setminus \Omega})\big)f(x)\\
&=\lim_{n\to\infty}\int_{\{t< \alpha_M(x)\}}  \exp\Big(-n\int^t_01_{M\setminus \Omega}(\IX_s(x)) ds\Big)\mathscr{A}^{\nabla}_V(x,t)\pa^{\nabla}_t(x)^{-1}f(\IX_t(x)) \Id\mathbb{P}.
\end{align*}
For paths from the set
$$
\{t< \alpha_M(x)\}=_{\IP}\{\IX_s(x)\in M\text{ for all $s\in [0,t]$}\}
$$
 one has 
$$
\int^t_01_{\hat{M}\setminus \Omega}(\IX_s(x)) ds=\int^t_01_{M\setminus \Omega}(\IX_s(x)) ds,
$$
and $\int^t_01_{M\setminus \Omega}(\IX_s(x)) ds=0$ is equivalent to $t<\beta_\Omega(x)$. Thus, $\IP$-a.s. in $\{t< \alpha_M(x)\}$ one has 
$$
\lim_{n\to\infty}\exp\Big(-n\int^t_01_{M\setminus \Omega}(\IX_s(x)) ds\Big)=1_{\{t< \beta_\Omega(x)\}}
$$
and 
\begin{align*}
&\lim_{n\to\infty}\int_{\{t< \alpha_M(x)\}}  \exp\Big(-n\int^t_01_{M\setminus \Omega}(\IX_s(x)) ds\Big)\mathscr{A}^{\nabla}_V(x,t)\pa^{\nabla}_t(x)^{-1}f(\IX_t(x)) \Id\mathbb{P}\\
&=\int_{\{t<\min(\beta_\Omega(x),\alpha_M(x) )\}} \mathscr{A}^{\nabla}_V(x,t)\pa^{\nabla}_t(x)^{-1}f(\IX_t(x)) \Id\mathbb{P}
\end{align*}
follows from dominated convergence, in view of (\ref{domii}).
\end{proof}

\subsection{Proof of Proposition \ref{main1}}

\begin{proof}[Proof of Proposition \ref{main1}]
With the preparations made, there is almost nothing left to prove: Note first that (\ref{poa}) is equivalent to

\begin{align}\label{hipo}
H_\Omega(\nabla,V)=\widetilde{H}_\Omega(\nabla,V).
\end{align}
\emph{(\ref{poa}) $\Rightarrow$ (\ref{aposssay}):} We have (\ref{hipo}), so that (\ref{aposssay}) follows from Proposition \ref{zbl}.\\
\emph{(\ref{aposssay}) $\Rightarrow$  (\ref{poa}):} 
(\ref{aposssay}) and Proposition \ref{zbl} imply
$$
\exp(-tH_\Omega(\nabla,V))=\exp(-t\widetilde{H_\Omega}(\nabla,V))\quad\text{ for all $t>0$,}
$$
showing (\ref{hipo}).

\end{proof}

\subsection{Proof of Proposition \ref{scal1}}

\begin{proof}[Proof of Proposition \ref{scal1}]

Note first that one has the Feynman-Kac formula 
$$
\exp(-tH_\Omega )f(x) = \int_{\{t<\alpha_\Omega\}}f(\IX_t) \Id\mathbb{P}^x
$$
for all $f\in L^2(\Omega)$, $t>0$,  $x\in \Omega$, meaning also that the right hand side is the smooth representative of $\exp(-tH_\Omega )f$. In case $\Omega$ is relatively compact with smooth boundary this is well-known. For the general case we can assume $f\geq 0$. Fix $x\in \Omega$ and exhaust $\Omega$ with a sequence of relatively compact smooth $\Omega_n$'s with $x\in \Omega_1$. Then the latter formula is valid with $\Omega$ replaced with $\Omega_n$, and we can take $n\to\infty$ using monotone convergence for integrals on the probabilistic side, noting that $\alpha_{\Omega_n}\nearrow \alpha_\Omega$ by Lemma \ref{dnaarrr} a), and the fact that (cf. p.213 in \cite{gri})
$$
\exp(-tH_{\Omega_n} )f(x)\nearrow \exp(-tH_{\Omega} )f(x).
$$
 Let us come to the actual proof of Proposition \ref{scal1}. Being equipped with the previously established results, we can follow the Euclidean proof from \cite{herbst} from here on:\\
As we have already remarked, the equivalence of (\ref{alkdd}) and (\ref{apossq}) follows immediately from Proposition \ref{main1}. \vspace{1mm}

\emph{(\ref{alkdd}) $\Rightarrow$ Kac-regularity:} (\ref{alkdd}) implies
\begin{align}\label{rtl2}
\widetilde{H}_\Omega=H_\Omega ,
\end{align}
by comparing the associated forms. In particular, the semigroups coincide and the Feynman-Kac formulae for the semigroups imply the first identity in
$$
0=\int(1_{\{t< \min(\beta_\Omega,\alpha_M) \}} -1_{\{t<\alpha_\Omega\}})f(\IX_t) \Id\mathbb{P}^x= \int_{\{ \alpha_\Omega\leq t< \beta_\Omega \}} f(\IX_t) \Id\mathbb{P}^x
$$
for all $f\in L^2(\Omega)$, $t>0$, $\mu$-a.e. $x\in \Omega$, where the second equality follows from
$$
1_{\{ \alpha_\Omega\leq t< \min(\beta_\Omega,\alpha_M) \}}=1_{\{t< \min(\beta_\Omega,\alpha_M) )\}}-1_{\{t<\alpha_\Omega\}\cap \{t< \min(\beta_\Omega,\alpha_M) \}}
$$
and $\alpha_\Omega\leq\beta_\Omega$ and $\alpha_\Omega\leq \alpha_M$. Letting $f$ run through $f_n:=1_{A_n}$, where $\{A_n:n\in\IN\}$ is a family of compact subsets of $M$ such that $A_n\subset A_{n+1}$ for all $n\in\IN$ and $\bigcup_{n\in\IN}A_n=M$, we get 
$$
\mathbb{P}^x\{ \alpha_\Omega\leq t< \min(\beta_\Omega,\alpha_M) \}=0\quad\text{ for all $t>0$, $\mu$-a.e. $x\in \Omega$}.
$$
Letting $t$ run through $\mathbb{Q}_{>0}$, 
\begin{align}\label{apaaoaaws}
\IP^{x}\{\alpha_\Omega< \min(\beta_\Omega,\alpha_M) \}=0\quad\text{ for $\mu$-a.e. $x\in \Omega$}.
\end{align}
In view of $\alpha_\Omega\leq\beta_\Omega$, the proof of the Kac-regularity of $\Omega$ is complete, once we have shown that the bounded function $h(x):= \IP^{x}\{\alpha_\Omega< \min(\beta_\Omega,\alpha_M) \}$ is continuous in $\Omega$. We are going to prove that for all open relatively compact subsets $D\subset \Omega$ with $\overline{D}\subset \Omega$ one has 
\begin{align}\label{apoaas}
\int h(\mathbb{X}_{\alpha_D}) d \IP^x =h(x)\text{ for all $x\in D$,}
\end{align}
which implies that $x\mapsto \IP^{x}\{\alpha_\Omega< \min(\beta_\Omega,\alpha_M) \}$ is harmonic and thus smooth in $\Omega$. To see (\ref{apoaas}), note first that $\alpha_D<\alpha_{\Omega}\leq \alpha_M$ $\IP^x$-a.s., as the transition density of Brownian motion is continuous. Now we can calculate
\begin{align*}
&\int h(\mathbb{X}_{\alpha_D})  d \IP^x\\
&=\int \IP^{\IX_{\alpha_D} }\{\alpha_\Omega< \min(\beta_\Omega,\alpha_M) \}d\IP^{x}\\
&= \int 1_{\{\alpha_\Omega< \min(\beta_\Omega,\alpha_M) \}}\big(\omega(\alpha_D(\omega)+\bullet)\big)d\IP^{x}(\omega)\\
&=  \int1_{\{\alpha_\Omega+\alpha_D< \min(\beta_\Omega,\alpha_M)+\alpha_D \}}(\omega)d\IP^{x}(\omega)\\
&=\int1_{\{\alpha_\Omega< \min(\beta_\Omega,\alpha_M) \}}(\omega)d\IP^{x}(\omega)\\
&=h(x),
\end{align*}
where the first inequality holds by definition, the second equality holds by the strong Markov property of Brownian motion, the third equality holds by elementary considerations, the fourth equality is trivial and and the last equality again holds by definition. \vspace{1mm}

\emph{Kac-regularity $\Rightarrow$ (\ref{alkdd}):} By the corresponding Feynman-Kac formulae we immediately find that Kac-regularity implies
$$
\exp(-t\widetilde{H}_\Omega)=\exp(-tH_\Omega )\quad\text{ for all $t>0$},
$$
so that $\widetilde{H}_\Omega=H_\Omega$.
\end{proof}

\subsection{Proof of Theorem \ref{main2} a)}

We will need the following covariant-Feynman-Kac formula:

\begin{Proposition}\label{feyn1} For every $\Omega\subset M$ open, $t\geq 0$, $f\in \Gamma_{L^2}(\Omega,\IEE)$, and $\mu$-a.e. $x\in M$ one has the covariant Feynman-Kac formula
\begin{align}\label{feyn2}
\mathrm{exp}\big(-t H_\Omega(\nabla,V)\big)f(x) =\int_{\{t<\alpha_\Omega(x) \}}   \mathscr{A}^{\nabla}_V(x,t)\pa^{\nabla}_t(x)^{-1}f(\IX_t(x))  \Id \mathbb{P}.  
\end{align}
\end{Proposition}


In case $\Omega$ is relatively compact with smooth boundary and $V$ is smooth, formula (\ref{feyn2}) is implicitly included in \cite{driver}. For general $\Omega$'s and $V$'s (\ref{feyn2}) seems to be new. 

\begin{proof}[Proof of Proposition \ref{feyn1}] It is enough to prove the formula in case $\Omega$ is relatively compact with a smooth boundary. Indeed, in the general case we can then pick a family of open subsets $\{\Omega_n:n\in\IN\}$ of $\Omega$ having a smooth boundary, such that $\Omega_n\subset \Omega_{n+1}$ for all $n\in\IN$ and $\bigcup_{n\in\IN}\Omega_n=\Omega$. Then we have $\alpha_{\Omega_n}(x)\nearrow\alpha_\Omega(x) $ and 
\begin{align}\label{feyn4}
\int_{\{t<\alpha_{\Omega_n}(x) \}} \mathscr{A}^{\nabla}_V(x,t)\pa^{\nabla}_t(x)^{-1}f(\IX_t(x))  \Id \mathbb{P}\to\int_{\{t<\alpha_{\Omega}(x) \}}  \mathscr{A}^{\nabla}_V(x,t)\pa^{\nabla}_t(x)^{-1}f(\IX_t(x))  \Id \mathbb{P}.  
\end{align}
follows from dominated convergence. Likewise,
\begin{align}\label{feyn5}
\mathrm{exp}\big(-t H_{\Omega_n}(\nabla,V)\big)f(x) \to \mathrm{exp}\big(-t H_\Omega(\nabla,V)\big)f(x)  
\end{align}
follows from applying Theorem \ref{monn2} and Remark \ref{aposv} thereafter: indeed, define the form $Q$ by 
$$
\dom(Q):=\Gamma_{W^{1,2}_0}(\Omega,E;\nabla,V),\quad Q(f_1,f_2)=\left\langle f_1,f_2\right\rangle_{\nabla,V,*},\quad f_1,f_2\in \dom (Q).
$$
Then the form $Q'$ from Theorem \ref{monn2} is closable (it has the closed extension $Q$), so that it remains to prove $\overline{Q'}=Q$, which is easily checked from the definitions.\\
Thus we can and we will assume $\Omega$ is relatively compact and has a smooth boundary. In case $V$ is smooth and bounded, the asserted formula is a simple consequence of Ito\rq{}s formula and parabolic regularity up to the boundary (cf. p. 102 in \cite{driver}). In case $V$ is bounded, we can use Friedrichs mollifiers to pick a sequence of smooth potentials $V_n:M\to \mathrm{End}(E)$ with $|V_n|\leq |V|$ and $|V_n-V|\to 0$ $\mu$-a.e. in $M$. Applying (\ref{feyn2}) with $V$ replaced with $V_n$ and taking $n\to \infty$ gives the result in this case. Finally, if $V$ is bounded from below, we can use the spectral calculus on the fibers of $E\to M$ to define a sequence of potentials 
$$
V_n:M\longrightarrow \mathrm{End}(E), \quad V_n(x):=\min(n,V(x))
$$
then apply (\ref{feyn2}) with $V$ replaced with $V_n$ and take $n\to \infty$. Each approximation argument can be justified precisely as in the proof of Theorem 2.11 in \cite{guneysu}, which treats the case $M=\Omega$: In each situation one uses convergence results for sesqulinear forms to control the left-hand site of (\ref{feyn2}) and convergence theorems for integrals to control the right-hand-side.
\end{proof}

\begin{proof}[Proof of Theorem \ref{main2} a)] There is nothing left to prove: We know from Proposition \ref{zbl} that 
$$
\lim_{n\to\infty}\mathrm{exp}\big(-t H_M(\nabla,V+n1_{M\setminus \Omega})\big)=  \mathrm{exp}\big(-t \widetilde{H}_{\Omega}(\nabla,V)\big)P_{\Omega}
$$
strongly (without any further assumptions on $\Omega$), and
$$
\widetilde{H}_{\Omega}(\nabla,V)=H_{\Omega}(\nabla,V)
$$
follows from a comparison of the covariant Feynman-Kac formulae for $\exp(-t\widetilde{H_\Omega}(\nabla,V))$ and $\exp(-tH_\Omega(\nabla,V))$ (cf. Proposition \ref{aposssy} and Proposition \ref{feyn1}), using that $\Omega$ is Kac regular.  
\end{proof}

\subsection{Proof of Theorem \ref{main2} b)}

We recall that given an open subset $\Omega\subset M$, the space $W^{1,2}(\Omega;d) $ is defined to be the space of all $f\in L^2(\Omega)$ such that $df\in \Gamma_{L^2}(\Omega;T^*M)$ in the sense of distributions. We are going to need a special case of the following result for the proof of Theorem \ref{main2} b):

\begin{Proposition}\label{bkjs}
\label{identification}
Let $\Omega\subset M$ be a relatively compact open subset with locally Lipschitz boundary. Then one has 
\begin{align*}
W^{1,2}_0(\Omega  ) = \big\{ f\in W^{1,2}(M ): f|_{M\setminus \Omega}=0 \>\>\text{\rm $\mu$-a.e.} \big\},
\end{align*}
in particular, $\Omega$ is Kac-regular.
\end{Proposition}

\begin{proof} Set $m:=\dim M$. It remains to prove 
$$
W^{1,2}_0(\Omega)\supset \{f\in W^{1,2}(M): f|_{M\setminus \Omega}= 0\>\text{$\mu$-a.e.}\}.
$$
According to Def. \ref{locallylip} this means that  for every $p\in \overline{\Omega}\setminus \Omega=\partial\Omega $ there exists a pair of open neighborhoods $V$, $X$ of $p$, and a  smooth diffeomorphism $\alpha: X\rightarrow A\subset \mathbb{R}^m$ such that: 
\begin{enumerate}
\item $\overline{V}\subset X$,
\item $\overline{V}$ is compact,
\item $\alpha(V\cap \Omega)$ is a bounded Lipschitz set of $\mathbb{R}^m$. 
\end{enumerate}
Let us consider a finite collection of open subsets of $M$, $\mathcal{W}:=\{V_1,...,V_q,W_1,...,W_n\}$, such that
\begin{itemize}
\item for each $j\in \{1,...,n\}$ we have $\overline{W_j}\subset \Omega$,
\item $\overline{\Omega}\subset V_1\cup...\cup V_q\cup W_1\cup...\cup W_n$,
\item for each $i\in \{1,...,q\}$ there exists another open subset $X_i$ such that the pair given by $V_i$ and $X_i$ satisfies  the conditions $(1)$--$(3)$ required in the statement.
\end{itemize}
Clearly such a finite collection of open subsets exists. Let $I$ be an open subset of $M$ such that $I\cap \Omega=\emptyset$ and 
$$
M=V_1\cup...\cup V_q\cup W_1\cup...\cup W_n\cup I.
$$
 Let 
$$
\mathcal{W'}:=\{V_1,...,V_q,W_1,...,W_n,I\}
$$
 and let 
$$
\{\phi_1,...,\phi_q,\psi_1,...,\psi_n,\tau\}
$$
be a partition of unity subordinated to $\mathcal{W}$. Given now $f\in W^{1,2}(M)$ with $f|_{M\setminus \Omega}= 0$ $\mu$-a.e., for each $i\in \{1,...,n\}$  we have $\supp(\psi_if)\subset W_i$. Therefore there exists a sequence $\{\beta^i_p\}_{p\in \mathbb{N}}\subset C^{\infty}_c(W_i)$ such that $\beta^i_p\rightarrow \psi_if$ in $W^{1,2}_0(W_i)$ as $p\rightarrow \infty$. Indeed, by a Meyers-Serrin type theorem \cite{guidetti}, we can pick a sequence $\{\overline{\beta}^i_p\}_{p\in \mathbb{N}}$  in $C^{\infty}(W_i)\cap W^{1,2}(W_i)$ such that $\overline{\beta}^i_p\rightarrow f$ in $W^{1,2}(W_i)$. Then, by defining $\beta^i_p:= \psi_i\overline{\beta}^i_p$, it is clear that $\{\beta^i_p\}_{p\in \mathbb{N}}\subset C^{\infty}_c(W_i)$ and  that $\beta^i_p\rightarrow \psi_if$ in $W^{1,2}_0(W_i)$ as $p\rightarrow \infty$. Let us now consider the other case. For each $i\in \{1,...,q\}$ let  $Y_i:=V_i\cap \Omega$. Then we have $\supp(\phi_if)\subset \overline{Y_i}\cap V_i$. Let $A_i:=\alpha_i(X_i)$, $E_i:=\alpha_i(V_i)$ and $B_i:=\alpha_i(Y_i)$.  Since  we assumed that $\overline{V_i}\subset X_i$ we have that  $(\alpha_i^*g_e)|_{V_i}$ is quasi-isometric to $g|_{V_i}$ where $g_e$ is the standard Euclidean metric on $\mathbb{R}^m$ and g the metric on $M$. In this way we can conclude that $(\phi_if)\circ (\alpha_i|_{V_i})^{-1}\in W_{0,e}^{1,2}(E_i)$ and that  $(\phi_if)\circ (\alpha_i|_{Y_i})^{-1}\in W^{1,2}_e(B_i)$, where of course $W^{1,2}_e$ stands for the various Sobolev spaces that are defined with respect to $g_e$. Furthermore we know  that $(\phi_if)\circ (\alpha_i|_{Y_i})^{-1}|_{A_i\setminus B_i}\equiv 0$ because $(\phi_if)|_{X_i\setminus Y_i}\equiv0$. Therefore, by extending $(\phi_if)\circ (\alpha_i|_{Y_i})^{-1}$ outside its support as the identically zero function, we can say, with a little abuse of notation, that $(\phi_if)\circ (\alpha_i|_{Y_i})^{-1}\in W^{1,2}_e(\mathbb{R}^m)$ and that $(\phi_if)\circ (\alpha_i|_{Y_i})^{-1}|_{\mathbb{R}^m\setminus B_i}\equiv 0$. As $B_i$ is a bounded Lipschitz open subset in $\mathbb{R}^m$ there exists a sequence $\{\upsilon^i_p\}_{p\in \mathbb{N}}\subset C_c^{\infty}(B_i)$ such that $\upsilon^i_p \rightarrow (\phi_if)\circ (\alpha_i|_{Y_i})^{-1}$ in $W^{1,2}_{0,e}(B_i)$ as $p\rightarrow \infty$. This latter property follows by the existence of the trace operator, see e.g. \cite{adams} Th. 5.37 or \cite{evans} Th. 4.6.  Finally, using the fact that $(\alpha_i^*g_e)|_{V_i}$ is quasi-isometric to $g|_{V_i}$, we can conclude that $\gamma^i_p\rightarrow \phi_if$ in $W^{1,2}_0(Y_i)$ as $p\rightarrow \infty$ where $\gamma^i_p:=\upsilon^i_p\circ (\alpha_i|_{Y_i})$ and clearly $\{\gamma^i_p\}_{p\in \mathbb{N}}\subset C^{\infty}_c(Y_i)$ by construction. Let us define now the following sequence of functions $\eta_p:=\gamma^1_p+\cdots+\gamma_p^q+\beta_p^1+\cdots+\beta_p^n$. It is clear by construction that $\{\eta_p\}_{p\in \mathbb{N}}\subset C^{\infty}_c(\Omega)$. Moreover we have
\begin{align} 
& \nonumber \|\eta_p-f\| =\|\gamma^1_p+\cdots+\gamma_p^q+\beta_j^1+\cdots+\beta_p^n-\sum_{i=1}^q\phi_if-\sum_{i=1}^n\psi_if\| \\
& \nonumber   \leq \|\gamma^1_p-\phi_1f\| +\cdots+\|\gamma^q_p-\phi_qf\| +\|\beta^1_p-\psi_1f\| +\cdots+\|\beta^n_p-\psi_nf\| 
\end{align}
and by construction all the terms in the second line tend to zero as $p\rightarrow \infty$. Hence we have shown that $\eta_p\rightarrow f$ in $\Gamma_{L^2}(\Omega,T^*_{\IC}M)$ as $p\rightarrow \infty$. Similarly we have 
\begin{align} 
& \nonumber \|d\eta_p-df\| =\|d\gamma^1_p+\cdots+d\gamma_p^q+d\beta_p^1+\cdots+d\beta_p^n-d(\sum_{i=1}^q\phi_if-\sum_{i=1}^n\psi_if)\| \\
& \nonumber  \leq \|d\gamma^1_p-d(\phi_1f)\| +\cdots+\|d\gamma^q_p-d(\phi_qf)\| +\|d\beta^1_p-d(\psi_1f)\| +\cdots\\
& \nonumber+\|d\beta^n_p-d(\psi_nf)\| 
\end{align}
and again by construction we know that all the terms on the right hand side of the inequality  tend to zero as $p\rightarrow \infty$. This tells us that $d\eta_p\rightarrow df$ in $\Gamma_{L^2}(\Omega,T^*_{\IC}M)$. In conclusion we have shown that $\eta_p\rightarrow f$ in $W^{1,2}(\Omega)$ and thereby we can conclude that $f\in W^{1,2}_0(\Omega)$ as desired.

\end{proof}

\begin{proof}[Proof of Theorem \ref{main2} b)]

Fix $x\in \Omega$. Pick a sequence $\Omega_n\subset \Omega$, $n\in\IN$ of relatively compact open subsets of $M$ with locally Lipschitz boundary, such that $x\in \Omega_1$, $\Omega_n\subset \Omega_{n+1}$ for all $n\in\IN$, $\bigcup_{n\in\IN} \Omega_n=\Omega$ and such that for all $n\in\IN$ there exists an open subset $\Upsilon_n\subset M$ such that $\Upsilon_n\cap \Omega=\Omega_n$ for all $n\in\IN$ and $\bigcup_{n\in\IN} (\Upsilon_n\cap \overline{\Omega})= \overline{\Omega}$. Then by Proposition \ref{bkjs} we have
$$
\mathbb{P}^x \{ \alpha_{\Omega_n} =\min(\beta_{\Omega_n},\alpha_M) \}=1\quad\text{for all $n$},
$$
so that using $\alpha_{\Omega_n}\to \alpha_{\Omega}$ and $\beta_{\Omega_n}\to \beta_{\Omega}$ $\mathbb{P}^x$-a.s. as $n\to \infty$, a consequence of Lemma \ref{dnaarrr}, we arrive at
\begin{align*}
\mathbb{P}^x \{ \alpha_{\Omega} \ne \min(\beta_{\Omega },\alpha_M) \}\leq \mathbb{P}^x\bigcup_{n\in\IN}\{ \alpha_{\Omega_n} \ne  \min(\beta_{\Omega_n},\alpha_M) \}\leq \sum_{n\in\IN} \mathbb{P}^x \{ \alpha_{\Omega_n} \ne   \min(\beta_{\Omega_n},\alpha_M)  \}=0,
\end{align*}
completing the proof.
\end{proof}

\section{Appendix: Some functional analytic facts}

In this section we collect some facts about the monotone convergence of nonnegative closed sesquilinear forms which are possibly not densely defined.

Let $\IHH$ be a complex separable Hilbert space   and let $\mathcal{L}(\IHH)$ be the space of bounded linear operators. Assume that $Q$ is a semibounded from below closed sesquilinear form on $\IHH$. Then $Q$ is a densely defined semibounded from below closed sesquilinear form on the Hilbert subspace $\IHH_{Q}:= \overline{\dom(Q)}^{\left\|\cdot\right\|_{\IHH}}\subset \IHH$ and thus there exists a unique semibounded from below self-adjoint operator $S_{Q}$ on $\IHH_{Q}$ such that
$$
\dom(\sqrt{S_{Q}})=\dom(Q),\quad Q(f_1,f_2)=\left\langle \sqrt{S_{Q}}f_1,\sqrt{S_{Q}}f_2\right\rangle_{\IHH}.
$$
Let then $P_{Q}:\IHH\to \IHH_{Q}$ denote the orthogonal projection. Recall also that for semibounded forms $Q_1$ and $Q_2$ in $\IHH$ one by sink investigation has $Q_1\leq Q_2$ if and only if $\dom(Q_1)\supset \dom(Q_2)$ and $Q_1(f,f)\leq Q_2(f,f)$ for all $f\in \dom (Q_2)$. Given a semibounded sesquilinear form $Q$ on $\IHH$ there exists a largest (with respect to '$\leq$') closable semibounded sesquilinear form $\mathrm{reg}(Q)$ on $\IHH$ (the so called \emph{regular part of $Q$}), such that $\mathrm{reg}(Q)\leq Q$.\\
The following results follow from Theorem 4.2 in \cite{simon}:

\begin{Theorem}\label{monn} Let $ Q_1\leq Q_2\leq\dots$ be a sequence of closed semibounded from below sesquilinear forms on $\IHH$.
Then 
$$
Q_{\infty}(f_1,f_2):=\lim_{n\to\infty} Q_n(f_1,f_2)
$$
with 
$$
\dom(Q_{\infty}):=\Big\{ f\in\bigcap_{n\in\IN}\dom(Q_n):\sup_{n\in\IN} Q_n(f,f)<\infty\Big\}\ and\ f_1,f_2\in \dom(Q_{\infty})
$$
is a closed semibounded from below sesquilinear form $Q_{\infty}$ in $\IHH$, and one has
$$
\mathrm{e}^{-t S_{Q_n}}P_{Q_n} \to \mathrm{e}^{ -t S_{Q_{\infty}}}P_{Q_{\infty}} \>\>\text{ strongly in $\ILL(\IHH)$ as $n\to\infty$, for all $t\geq 0$.}
$$ 
\end{Theorem}



\begin{Theorem}\label{monn2} Let $Q_1\geq Q_2\geq\dots$ be a sequence of closed semibounded sesquilinear forms on  $\IHH$. Define a nonnegative sesquilinear form in $\IHH$ given by
$$
Q'(f_1,f_2):=\lim_{n\to\infty} Q_n(f_1,f_2),\quad\dom(Q'):=\bigcup_{n\in\IN} \dom(Q_n), \quad f_1,f_2\in \dom(Q')
$$
and let $Q_{\infty}$ denote the form on $\IHH$ given by the closure of $\mathrm{reg}(Q\rq{})$ in $\IHH$. Then one has
$$
\mathrm{e}^{-t S_{Q_n}}P_{Q_n}\to \mathrm{e}^{-t S_{Q_{\infty}}}P_{Q_{\infty}}\text{ strongly in $\ILL(\IHH)$ as $n\to\infty$, for all $t\geq 0$ .}
$$ 
\end{Theorem}

\begin{Remark}\label{aposv} Note that in the above situation $Q_{\infty}$ is the closure of $Q\rq{}$ if $Q\rq{}$ is closable, and $Q_{\infty}=Q\rq{}$ if $Q\rq{}$ is even closed.
\end{Remark}


\begin{thebibliography}{99}

\bibitem{adams}
R{.} A{.} Adams, J{.} J{.} F{.} Fournier.
\newblock Sobolev Spaces. Second edition.
\newblock Pure and Applied Mathematics (Amsterdam), 140. Elsevier/Academic Press, Amsterdam, 2003

\bibitem{aubin}
T{.} Aubin.
\newblock Espaces de Sobolev sur les vari\'et\'es riemanniennes. 
\newblock Bull{.} Sci{.} Math{.} (2) 100 (1976), no{.} 2, 149--173.



\bibitem{driver} B{.} K{.} Driver, A{.} Thalmaier.
\newblock Heat equation derivative formulas for vector bundles.
\newblock J{.} Funct{.} Anal{.} 183 (2001), no{.} 1, 42--108. 


\bibitem{daco}
B{.} Dacorogna.
\newblock Introduction to the Calculus of Variations. 
\newblock Imperial College Press, London, 2004.

\bibitem{evans}
L{.} C{.} Evans, R{.} F{.} Gariepy. 
\newblock Measure theory and fine properties of functions. Revised edition. 
\newblock  CRC Press, Boca Raton, FL, 2015.

\bibitem{fuku}
 M{.} Fukushima, Y{.} Oshima,  M{.} Takeda.
\newblock Dirichlet forms and symmetric Markov processes. 
\newblock De Gruyter Studies in Mathematics, 19{.} Walter de Gruyter \& Co{.}, Berlin, 1994.





\bibitem{grio}
 A. Grigor'yan.
 \newblock Heat kernels on weighted manifolds and applications{.} The ubiquitous heat kernel.
 \newblock  Contemp{.} Math{.}, 398, 93--191, Amer{.} Math{.} Soc{.}, Providence, RI, 2006. 

\bibitem{gri}  A. Grigoryan.
 \newblock Heat kernel and analysis on manifolds.
\newblock AMS/IP Studies in Advanced Mathematics, 47{.} American Mathematical Society, Providence, RI; International Press, Boston, MA, 2009.




\bibitem{guneysu}
B{.} G\"uneysu. 
\newblock On generalized Schr\"odinger semigroups.
\newblock J{.} Funct{.} Anal{.} 262 (2012), no{.} 11, 4639--4674. 

\bibitem{guneysu2}
B{.} G\"uneysu. 
\newblock Nonrelativistic hydrogen type stability problems on nonparabolic $3$-manifolds. 
\newblock Ann{.} Henri Poincare 13 (2012), no{.} 7, 1557--1573.


\bibitem{gunbook}
B{.} G\"uneysu. 
\newblock Covariant Schr\"odinger semigroups on Riemannian manifolds. 
\newblock Operator Theory: Advances and Applications, 264{.} Birkh\"ouser/Springer, Cham, 2017{.} xviii+239 pp. 

\bibitem{hedberg}
I{.} L{.} Hedberg.
\newblock Spectral synthesis and stability in Sobolev spaces. 
\newblock Euclidean harmonic analysis (Proc. Sem{.}, Univ{.} Maryland, College Park, Md{.}, 1979), pp{.} 73--103, Lecture Notes in Math{.}, 779, Springer, Berlin, 1980.


\bibitem{herbst}
 I{.} W{.} Herbst, Z{.} X{.} Zhao.
\newblock  Sobolev spaces, Kac-regularity, and the Feynman-Kac formula. 
\newblock Seminar on Stochastic Processes, 1987 (Princeton, NJ, 1987), 171--191, Progr. Probab. Statist., 15, Birkh\"auser Boston, Boston, MA, 1988.


\bibitem{guidetti} 
D{.} Guidetti, B{.} G\"uneysu, D{.} Pallara.
\newblock $L^1$-elliptic regularity and $H=W$ on the whole $L^p$-scale on arbitrary manifolds.
\newblock Ann{.}  Acad{.}  Sci{.}  Fenn{.}  Math{.}  42 (2017), no{.}  1, 497--521.

\bibitem{GuPo}
V{.} Guillemin, D{.} Pollack.
\newblock Differential topology.
\newblock  Prentice-Hall, Inc., Englewood Cliffs, N.J., 1974.

\bibitem{MoHi}
M{.} W{.} Hirsch.
\newblock Differential topology. 
\newblock Corrected reprint of the 1976 original. Graduate Texts in Mathematics, 33. Springer-Verlag, New York, 1994


\bibitem{hsu} 
E{.} P{.} Hsu.
\newblock Stochastic analysis on manifolds.
\newblock Graduate Studies in Mathematics, 38{.} American Mathematical Society, Providence, RI, 2002.




\bibitem {lawson}
H{.} B{.} Jr{.} Lawson, M{.}-L{.} Michelsohn.
\newblock  Spin geometry. 
\newblock Princeton Mathematical Series, 38{.} Princeton University Press, Princeton, NJ, 1989.





\bibitem{nico}  
L.I. Nicolaescu.
\newblock Lectures on the geometry of manifolds{.} Second edition. 
\newblock World Scientific Publishing Co{.} Pte{.} Ltd{.}, Hackensack, NJ, 2007.



\bibitem{pigo}
L{.} F{.} Pessoa, S{.} Pigola, A{.} G{.} Setti. 
Dirichlet parabolicity and L$^{1}$--Liouville property under localized geometric conditions.
\newblock J{.} Funct{.} Anal{.} 273 (2017), 652--693.


\bibitem{simon}
 B{.} Simon.
\newblock A canonical decomposition for quadratic forms with applications to monotone convergence theorems.
\newblock J{.} Funct{.} Anal{.} 28 (1978), no{.} 3, 377--385.


\bibitem{stein}
E{.} M{.} Stein.
\newblock Singular Integrals and Differentiability of Functions, Princeton University.
\newblock Press, Princeton, New Jersey, 1970.


\bibitem{stollmann1}
P{.} Stollmann.
\newblock Formtechniken bei Schr\"odingeroperatoren. 
\newblock Diploma thesis, Munich 1985. 

\bibitem{stollmann2}
P{.} Stollmann, J{.} Voigt.
\newblock Perturbation of Dirichlet forms by measures. 
\newblock Potential Anal{.} 5 (1996), no{.} 2, 109--138. 


\bibitem{stroock}
D{.} Stroock.
\newblock The Kac approach to potential theory: Part I. 
\newblock J. Mathematics and Mechanics. vol. 16. no. 8. 829--852 (1967).



\end{thebibliography}
\end{document}